\documentclass[a4paper,11pt,pdf]{amsart}
\usepackage{enumerate, amsmath, amsfonts, amssymb, amsthm,  wasysym, graphics, graphicx, xcolor, frcursive,bbm}
\definecolor{darkblue}{rgb}{0.0,0,0.7} 
\definecolor{darkred}{rgb}{0.7,0,0} 
\usepackage[all]{xy}

\numberwithin{equation}{section}
\newtheorem{thm}[equation]{Theorem}

\newtheorem{prop}[equation]{Proposition}

\newtheorem{lem}[equation]{Lemma}

\theoremstyle{definition}

\newtheorem{rmq}[equation]{Remark}

\newcommand\lexp[2]{\kern\scriptspace\vphantom{#2}^{#1}\kern-\scriptspace#2}

\title[On generalized categories of Soergel bimodules in type $A_2$]{On generalized categories of Soergel bimodules in type $A_2$}

\author{Thomas Gobet}
\address{Thomas Gobet, Institut \'Elie Cartan de Lorraine, Universit\'e  de Lorraine, site de Nancy,
B.P. 70239,
54506 Vandoeuvre-l\`{e}s-Nancy Cedex, France}
\email{thomas.gobet@univ-lorraine.fr}
\author{Anne-Laure Thiel}
\address{Anne-Laure Thiel, Institut f\"ur Geometrie und Topologie, Fachbereich Mathematik, Universit\"at Stuttgart, 70569 Stuttgart, Deutschland.}
\email{anne-laure.thiel@mathematik.uni-stuttgart.de} 

\begin{document}
\maketitle

\begin{abstract}
In this note, we compute the split Grothendieck ring of a generalized category of Soergel bimodules of type $A_2$, where we take one generator for each reflection. We give a presentation by generators and relations of it and a parametrization of the indecomposable objects in the category, by realizing them as rings of regular functions on certain unions of graphs of group elements on a reflection faithful representation.  
\end{abstract}
\section{Introduction}

Let $(W,S)$ be a Coxeter system, $V$ a reflection faithful representation (as defined in \cite[Definition 1.5]{S}) of $W$ over the real numbers. Let $R=\mathcal{O}(V)=S(V^*)$ be the ring of polynomial functions on $V$, graded so that $\mathrm{deg}(V^*)=2$. To each simple reflection $s\in S$, one associates the graded $R$-bimodule $B_s:=R\otimes_{R^s} R(1)$ (here $(1)$ denotes a grading shift), where $R^s\subseteq R$ is the graded subring of $s$-invariant functions. Soergel showed that the split Grothendieck ring $\langle \mathcal{B} \rangle$ of the Karoubian envelope $\mathcal{B}$ of the category of graded $R$-bimodules generated by (shifted) tensor products of the $B_s$ over $R$ is isomorphic to the Iwahori-Hecke algebra $\mathcal{H}(W)$ of the Coxeter system $(W,S)$ (\cite[Theorem 1.10, Remark 1.14]{S}). The class $\langle B_s\rangle$ of $B_s$ in $\langle \mathcal{B} \rangle$ corresponds to the Kazhdan-Lusztig generator $C'_s$ of $\mathcal{H}(W)$, see~\cite{KL}. More generally, Soergel's conjecture \cite[Conjecture 1.13]{S} (which was proven in full generality by Elias and Williamson \cite{EW}) asserts that to each element $C_w'$ of the canonical Kazhdan-Lusztig basis corresponds an indecomposable bimodule in $\mathcal{B}$. The bimodule corresponding to $w$ generalizes the (equivariant) intersection cohomology of the Schubert variety $X_w$ (which can be defined only in the case where $W$ is a Weyl group; see \cite{S1}). 

The definition of $B_s$ makes sense for every reflection $s\in T:=\bigcup_{w\in W} w S w^{-1}$, not only for the simple reflections $s\in S$. In this note, we compute the split Grothendieck ring $\langle\mathcal{C}\rangle$ of the category $\mathcal{C}$ generated by the $B_t:=R \otimes_{R^t} R(1)$, $t\in T$ in the case where $W$ is of type $A_2$ (note that in Soergel's category $\mathcal{B}$, the notation $B_t$ stands for the indecomposable bimodule associated to the group element $t$, which in general is \textit{not} isomorphic to $R\otimes_{R^t} R(1)$; in this note $B_t$ will always denote $R\otimes_{R^t} R(1)$). This ring is a free $\mathbb{Z}[v^{\pm 1}]$-algebra $A(W)$ of rank $20$, and we identify the indecomposable objects in $\mathcal{C}$: they are in one-to-one correspondence with subsets $A\subseteq W$ such that there is $t\in T$ with $tA=A$. We give a presentation by generators and relations of the resulting algebra $A(W)$, which turns out to be a quotient of the Iwahori-Hecke algebra of type $\widetilde{A_2}$. While the category $\mathcal{C}$ can be defined for an arbitrary Coxeter group $W$, we do not know how to compute $\langle \mathcal{C}\rangle$ for $W$ of type $A_3$ or $B_2$, where it is not even clear to us that $\langle \mathcal{C} \rangle$ has finite rank.

It would be interesting to find a well-behaved generalization of the algebra obtained in type $A_2$ for finite $W$ or in type $A_n$, possibly by considering split Grothendieck rings of suitable subcategories of $\mathcal{C}$ containing $\mathcal{B}$.

\section{General facts}

Let $x\in W$ and let $R_x$ denote the $R$-bimodule $R$ with the right operation twisted by $x$, that is, for $a\in R$, $r\in R$ we set $a\cdot r= ax(r)$. Note that the embedding $V\hookrightarrow V\times V, v\mapsto (v, x^{-1}v)$ induces an isomorphism of graded bimodules $R_x\cong \mathcal{O}(\mathrm{Gr}(x))$, where $\mathrm{Gr}(x):=\{ (xv, v)~|~v\in V\}$ and $\mathcal{O}(-)$ denotes the $\mathcal{R}$-algebra of regular functions. We have $R_x\otimes_R R_y\cong R_{xy}$ for all $x,y\in W$. Note the following isomorphisms: 
\begin{eqnarray}\label{eq:1}
R\otimes_{R^t} R_t \cong R\otimes_{R^t} R \cong R_t\otimes_{R^{t}} R, \forall t\in T,\\
\label{eq:2} R_w B_t \cong B_{wtw^{-1}} R_w, \forall t\in T, w\in W.
\end{eqnarray} 
The first isomorphisms are given by the maps $a\otimes b\mapsto a\otimes t(b)$ and $a\otimes b\mapsto t(a)\otimes b$ respectively. Hence $B_t\otimes_R R_t\cong B_t\cong R_t\otimes_R B_t$. The last one is given by the well-defined invertible map $R_w \otimes_{R^t} R\rightarrow R\otimes_{R^{wtw^{-1}}} R_w$, $a\otimes b\mapsto a\otimes w(b)$. 

For simplicity we will denote tensor products over $\otimes_R$ by juxtaposition. A consequence of the above isomorphisms is that for all $s, t_1, \dots, t_k\in T$, one has isomorphisms of graded $R$-bimodules 
\begin{eqnarray}\label{eq:3}
B_s B_{t_1} B_{t_2} \cdots B_{t_k} B_s \cong B_s B_{st_1 s} B_{s t_2 s} \cdots B_{st_ks} B_s. 
\end{eqnarray}

We denote by $\mathcal{C}$ the category obtained as Karoubi envelope of the category of (shifted) tensor products of $B_t$ for $t\in T$. We denote by $\mathcal{C}^{\mathrm{ext}}$ the category generated by (shifted) tensor products of $B_t$ for $ t\in T$ and $R_w$ for $ w\in W$. Note that one has inclusions as full subcategories $\mathcal{B}\subseteq \mathcal{C}\subseteq \mathcal{C}^{\mathrm{ext}}$. By \eqref{eq:2}, observe also that $\mathcal{C}^{\mathrm{ext}}$ is generated by the bimodules $B_s$ for $ s\in S$ and $R_w,$ for $w\in W$. 

The isomorphisms given in \eqref{eq:3} give a family of relations satisfied by the generators of $\mathcal{C}$, in addition to those which hold in $\mathcal{B}\subseteq\mathcal{C}$. We assume the reader to be familiar with the combinatorics of Soergel's category $\mathcal{B}$. 

Given $A\subseteq W$, we write $R(A)$ for the graded $R$-bimodule $\mathcal{O}\left(\bigcup_{x\in A} \mathrm{Gr}(x)\right)$. Note that $\bigcup_{x\in A} \mathrm{Gr}(x)$ is a closed subscheme of $V\times V$, inducing a surjective map of graded $R$-bimodules $\mathcal{O}(V\times V)\cong R\otimes_{\mathbb{R}} R\twoheadrightarrow R(A)$. It implies that $R(A)$ is generated as a graded $R$-bimodule by any nonzero element in its degree zero component, hence that it is indecomposable, since this component is one-dimensional. By convention we set $R(\emptyset)=0$. 

We have $B_t\cong R(\{e, t\})(1)$ (see~\cite[Remark 4.3]{S}). For more on the properties of rings of regular functions on unions of graphs we refer the reader to~\cite[Section 4.3]{Will}.   

For $w\in W$, $A\subseteq W$ we have 
\begin{eqnarray}\label{eq:4}
R_w\otimes_R R(A)\cong R(wA)\; \text{and} \; R(A)\otimes_R R_w\cong  R(Aw),
\end{eqnarray}  

where the first map is given by $a\otimes b\mapsto \{(u,v)\mapsto a(u) b\left(w^{-1}u,v\right)\}$ and the second one by $a\otimes b\mapsto \{(u,v)\mapsto a\left(u, wv\right) b(wv)\}$. 

\begin{lem}[{\cite[Lemma 4.5 (1)]{S}}]\label{stable}
Let $A\subseteq W$, $t\in T$ such that $tA=A$. Then $R\otimes_{R^t} R(A)\cong R(A)\oplus R(A)(-2)$. 
\end{lem}

From now we assume $(W, S)$ to be dihedral, that is, that $|S|=2$. In that case there holds the following Lemma, which will allow us to parametrize the indecomposable bimodules in $\mathcal{C}$ in case $(W,S)$ is of type $A_2=I_2(3)$.

\begin{lem}[{\cite{S}}]\label{Soergel} 
Let $A\subseteq W$, $t,s\in T$ such that $t A=A$ and $|A\backslash \left(A\cap sA\right)|=2$. Then 
$$R\otimes_{R^s} R(A)\cong R\left(A\cup sA\right)\oplus R\left(A\cap sA\right)(-2).$$
\end{lem}

\begin{proof}
One can give exactly the same proof as in \cite[Proposition 4.6]{S}: in the
proof there, one considers sets $A$ of the form $A = \{y\in W~|~y \leq x\}$ where $\leq$
is the Bruhat order and uses the fact that $A\backslash \left(A\cap sA\right)=\{x, t'x\}$ for some $t'\in T$ with $t'\neq s$ (Soergel assumes $s\in S$). But by looking at the proof one sees that it
works for any $s\in T$ and any $A$ such that $A\backslash (A\cap sA)$ is a $t'$ -stable subset of cardinality 2
for some reflection $t'\neq s$, which can be easily observed in our case: since $A= tA$, we have $|A|=2k$ for some $k\geq 1$ and exactly half of the elements in $A$ have odd length with respect to the generating set $S$ of $W$ (equivalently, are reflections). But since $A\cap sA$ is $s$-stable, the same holds for $A\cap sA$, implying that among the two element $x,y$ of $A\backslash (A\cap sA)$, exactly one, say $x$, has odd length, while $y$ has even length. It implies that $t':=yx^{-1}$ has odd length, hence is a reflection, which shows the claim. 
\end{proof}

\section{Parametrizing the indecomposable objects in type $A_2$}

In this section, we assume $(W,S)$ to be of type $A_2$. The aim of this section is to identify the indecomposable bimodules in $\mathcal{C}$ in this case, using Lemmas~\ref{stable} and \ref{Soergel}. To this end, consider the set $$\mathcal{X}:=\{ \emptyset\neq A\subseteq W~|~\exists t\in T: t A=A\}.$$

There are $19$ elements belonging to $\mathcal{X}$, given by $W$, $3$ $t$-stable subsets of cardinality $2$ for each $t\in T=\{ t_1, t_2, t_3\}$, and their complements. 

\begin{prop}\label{prop:indec}
The unshifted indecomposable bimodules in $\mathcal{C}$ are given by $R(A)$, $A\in \mathcal{X}$, and $R$. In particular, the $\mathbb{Z}[v^{\pm 1}]$-algebra $A(W):=\langle \mathcal{C} \rangle$ (where $v$ acts by a grading shift $(1)$) has rank $20$ as a free $\mathbb{Z}[v^{\pm 1}]$-module.
\end{prop}

Before establishing the above Proposition we prove:

\begin{lem}\label{tech}
Let $A\in\mathcal{X}$, $t\in T$. Then either $t A=A$, or $A\backslash (A\cap tA)$ is a $t'$-stable subset of cardinality $2$ for some $t'\in T$. 
\end{lem}

\begin{proof}
Assume that $tA\neq A$. By definition of $\mathcal{X}$, there is $s\in T$ such that $s A=A$. If $|A|=2$, then since $s\neq t$ we have $A\cap t A=\emptyset$ and we are done. If $|A|=4$, then $|A\cap tA|=2$ since $A\neq tA$ and $|W|=6$. 

Write $A=\{x, sx, y, sy\}$ and assume without loss of generality that $x\in A\cap tA$. If $x=ty$, then $y = tx$ and $A\backslash \left(A\cap tA\right)=\{sx, sy\}$ is $sts$-stable as $(sts)sx=sy$. If $x=tsx$, then $s=t$ and $A$ is $t$-stable, a contradiction. If $x=tsy$, then $sy=tx$ and $A\backslash (A\cap tA)=\{sx, y\}$ is $sts$-stable as $(sts)sx=y$.  
\end{proof}

As an immediate Corollary of \ref{stable}, \ref{tech}, \ref{Soergel}, we have that for all $t\in T$ and $A\in\mathcal{X}$, the tensor product $B_t\otimes_R R(A)$ decomposes as a direct sum of $R(B)$, $B\in\mathcal{X}$. Together with the fact that $B_t\cong R(\{e, t\})(1)$ for all $t \in T$, we deduce that $\mathcal{C}$ has at most $20$ unshifted indecomposable bimodules, and it remains to prove that $R(A)$ occurs as a an indecomposable bimodule of $\mathcal{C}$ for every $A\in\mathcal{X}$. 

\begin{proof}[Proof of Proposition~\ref{prop:indec}]
We have to show that for all $A\in\mathcal{X}$, $R(A)$ occurs as a direct summand of (a shift of) $B_{t_1} B_{t_2}\cdots B_{t_k}$ for some $t_i\in T$. Since $A \in \mathcal{X}$, let us fix $t \in T$ such that $tA = A$.

Assume that $|A|=2$. If $A=\{e, t\}$, then $R(A)=B_t(-1)$ in which case we are done. If $A\neq\{e, t\}$, then $A$ is necessarily of the form $A=\{t_1, tt_1\}$ for $t, t_1\in T$, $t\neq t_1$ (in particular, we have $tt_1t=t_1 t t_1$). Applying twice Lemma~\ref{Soergel} we have first $B_{t_1} B_{t_1 t t_1}\cong R\left(\{t_1, t_1tt_1, e, t t_1\}\right)(2)$ and then
$$B_t R\left(\{ t_1, t_1tt_1, e, tt_1\}\right)\cong R(W)(1)\oplus R\left(\{t_1, tt_1\}\right)(-1),$$
which shows that $R(A)$ and $R(W)$ both appear as a summand of a shift of $B_t B_{t_1} B_{t_1 t t_1}$.

Assume that $|A|=4$. Write $A=\{tx, x, ty, y\}$. If $e\in A$, then without loss of generality we can assume that $x=e$ and $y=t_1\in T$ with $t_1 \neq t$. We then have $R(A)\cong B_t B_{t_1}(-2)$ by Lemma~\ref{Soergel}. If $e\notin A$ then we can assume without loss of generality that $x=t_1\in T$ and $y= t t_1 t$. We then have by Lemma~\ref{Soergel} that $R(A)\cong B_t\otimes R(\{t_1, t_1t\})(-1)$. But, as already proved, the bimodule associated to the $(t_1tt_1)$--stable set $\{t_1, t_1t\}$ of cardinality $2$ has to appear as a summand of $B_{t_1 t t_1} B_{t_1} B_t(-1)$, hence $R(A)$ appears as a summand of $B_t B_{t_1 t t_1} B_{t_1} B_t(-2)$.     
\end{proof}

By Relation~\eqref{eq:4}, tensoring a $B_t$ by an $R_w$ gives an indecomposable bimodule $R(A)$ which by Proposition~\ref{prop:indec} lies in $\mathcal{C}$. As a consequence we get:

\begin{lem}
The unshifted indecomposable bimodules in $\mathcal{C}^{\mathrm{ext}}$ are given by the unshifted indecomposable ones in $\mathcal{C}$ and the $R_w$ for $w\in W\backslash\{e\}$. In particular, the split Grothendieck ring $\langle \mathcal{C}^{\mathrm{ext}}\rangle$ has rank $25$. 

\end{lem}

\section{A presentation by generators and relations}

In this section, we give a presentation by generators and relations of $A(W)=\langle \mathcal{C} \rangle$ in type $A_2$. Let $(W,S)$ be of type $A_2$, with set of reflections $T=\{t_1, t_2, t_3\}$.  

\begin{thm}
The algebra $A(W)$ is generated as $\mathbb{Z}[v^{\pm 1}]$-algebra by $C_i$, $i=1,2,3$ with relations
\begin{enumerate}
\item $C_ i^2=(v+v^{-1}) C_i,~\forall i=1,2,3,$
\item $C_i C_j C_i + C_j = C_i + C_j C_i C_j,~\forall i\neq j,$
\item $C_i C_j C_i= C_i C_k C_i,~\text{if }\{i,j,k\}=\{1,2,3\}$,
\item $C_i C_j C_k C_i=C_i C_k C_j C_i,~\text{if }\{i,j,k\}=\{1,2,3\}.$
\end{enumerate}
For all $i$ we have $C_i=\langle B_{t_i}\rangle$ and $R(1)=v$. 
\end{thm}

\begin{rmq} Note that the algebra defined above is a quotient of the Hecke algebra $\mathcal{H}( W_{\widetilde{A_2}})$ of the affine Weyl group of type $\widetilde{A_2}$, as the Kazhdan-Lusztig presentation of this Hecke algebra has three generators satisfying precisely the Relations $(1)$ and $(2)$ above. 
\end{rmq}

\begin{proof}
We first show that the above relations are satisfied in $\mathcal{C}$: the first relation is a consequence of Lemma~\ref{stable}. The second relation holds in Soergel's category $\mathcal{B}$ in case $\{ t_i, t_j\}=S$ but it generalizes here using either Lemma~\ref{Soergel} or by conjugating the relation in Soergel's category by an $R_s$, $s\in S$ and using Relations~\eqref{eq:1}--\eqref{eq:3}. The last two relations are just a particular case of Relation~\eqref{eq:3}. This shows that $\langle \mathcal{C} \rangle$ is a quotient of $A(W)$. Since by Proposition~\ref{prop:indec} we know that $\langle\mathcal{C}\rangle$ is free of rank $20$, it suffices to show that the algebra defined by the above presentation is $\mathbb{Z}[v^{\pm 1}]$-linearly spanned by a set of $20$ elements. We show that every monomial in the $C_i$'s can be expressed as a linear combination of the $20$ elements $1$, $C_1$, $C_2$, $C_3$, $C_1 C_2$, $C_2 C_1$, $C_2C_3$, $C_3C_2$, $C_1 C_3$, $C_3 C_1$, $C_1 C_2 C_1$, $C_i C_j C_k$ with $\{i,j,k\}=\{1,2,3\}$, $C_1 C_2 C_3 C_1$, $C_2 C_1 C_3 C_2$ and $C_3 C_1 C_2 C_3$. To this end, it suffices to show that any word of length at most $5$ in the $C_i$'s is a linear combination of these $20$ elements. 

For words of length at most two it is clear since $C_i^2=(v+v^{-1})C_i$. Given a word $C_i C_j C_k$ of length $3$, if $|\{i,j,k\}|=3$ then our word is an element of the above list. If two consecutive letters are the same, then we are done using again the quadratic relation. If $i=k, j\neq i$, then using the second and third relation we can express our word as a linear combination of $C_1 C_2 C_1$ and $C_1, C_2, C_3$. Now consider a word $C_i C_j C_k C_\ell$ of length $4$. If $i=\ell$, then if $|\{i,j,k\}|=3$, up to permuting the two middle letters using the fourth relation our word belongs to the above list. If $|\{i,j,k\}|\neq 3$, then two consecutive letters in the word have to agree, hence we can apply a quadratic relation to express our word as a linear combination of words of length three for which we already shown the result. Assume that $i\neq \ell$. Then either a quadratic relation can be applied, or the word is $C_i C_j C_i C_\ell$ or $C_i C_\ell C_k C_\ell$, with $|\{i,j,\ell\}|=3$ (resp. $|\{i,k,\ell\}|=3$). Assume that the word is $C_i C_j C_i C_\ell$, the other case is symmetric. Then applying the third relation we get $C_i C_\ell C_i C_\ell$ which using the first two relations can be expressed as a linear combination or words of length at most three. Now if $C_i C_j C_k C_\ell C_m$ is a word of length five, then we can assume that $\ell=i$, otherwise we already saw that the word $C_i C_j C_k C_\ell$ is a linear combination of words of length at most three: it implies that $C_i C_j C_k C_\ell C_m$ is a linear combination of words of length at most $4$, for which we already know the result. If $m=i$ then the word is $C_i C_j C_k C_i^2$ which can be reduced to a linear combination of words of length four by applying the quadratic relation. Note that we can assume $\{i,j,k\}=3$, otherwise we already saw that the word $C_i C_j C_k C_i$ can be expressed as a linear combination of words of length at most three. Assume that $m=k$, the case $m=j$ being similar after appling the relation $C_i C_j C_k C_i=C_i C_kC_jC_i$. Thanks to the third relation we have $$C_i C_j C_k C_i C_k= C_i C_j C_k C_j C_k,$$
and using the first two relations the word $C_j C_k C_j C_k$ can be expressed as a linear combination of words of smaller length. Hence $C_i C_j C_k C_i C_k$ can be expressed as a linear combination of words of length at most four, for which we have already shown the property. 
 
\end{proof}

\begin{rmq}\label{rmk:comb}
It can be observed by straightforward computations that any tensor product of the generators of $\mathcal{C}^{\mathrm{ext}}$, in case $W$ is of type $A_2$, is isomorphic to a bimodule of the form
\begin{equation}
R_w B_{s_1} \cdots B_{s_k} R_{w'}
\label{twBStw}
\end{equation}
with $w,w' \in W$ and $s_1, \cdots s_k$ simple reflections. Since the indecomposable objects in $\mathcal{B}$ are fully understood and tensoring with $R_w$ defines an invertible functor, the classification of indecomposables in $\mathcal{C}$ which we made in the previous section can also be derived from this fact (but the proof with rings of regular functions appears as more conceptual to us). In Section~\ref{sec:rmks}, we show that this property is not fulfilled in the category $\mathcal{C}^{\mathrm{ext}}$ attached to other Coxeter groups, by providing explicit counterexamples. 

\end{rmq}

\section{Some remarks about other Coxeter groups}\label{sec:rmks}

The Coxeter group of type $A_2$ is the smallest Coxeter group with a braid relation, and in many cases the situation in type $A_2$ gives a hint of a more general situation. In this section, we give a few examples observed in bigger groups, which show that the strategy used in type $A_2$ is too naive to work in higher rank and even for other dihedral groups. We actually do not know what the correct generalization of the algebra $A(W)$ should be. In type $A_2$, one can think of the added generator $B_t$ where $t$ is the non simple reflection as being associated to the highest root of $W$ viewed as a Weyl group, but we are not even able to describe the Grothendieck ring of the category generated by the $B_s, s\in S$ and $B_t$, where $t$ is the reflection associated to the highest root, in other cases.  

\subsection{Type $B_2$}

In Soergel's category $\mathcal{B}$, when $W$ is a dihedral group then every unshifted indecomposable bimodule is isomorphic to an $R(A)$ (see~\cite[Section 4]{S}). We have seen above that this stays true in $\mathcal{C}$ in the case where $W$ is of type $A_2$. We give a counterexample of this fact for the category $\mathcal{C}$ in the case where $W$ is of type $B_2$. Let $W$ be such a Coxeter group with $S=\{s, t\}$. Consider the bimodule $B:= B_{tst} B_s B_t$. Note that $B_s B_t\cong R(\{e,s,t,st\})(2)$ is a Soergel bimodule, but we cannot apply Lemma~\ref{Soergel} to decompose $B$. As a Soergel bimodule $B_s B_t$ possesses a standard and a costandard filtration (as defined in~\cite[Section 5]{S}). There is a short exact sequence $$0\longrightarrow R_{tst}(-1)\longrightarrow B_{tst} \longrightarrow R(1)\longrightarrow 0.$$
Tensoring this short exact sequence by $B_s B_t$ we see, using the isomorphism theorems, that $B$ has a filtration where (a shifted copy of) every $R_x$, $x\in W$, appears exactly once as subquotient (in fact, using twisted filtrations of $B_s B_t$ as in \cite{Twisted}, it is not difficult to show that $B$ has both a standard and a costandard filtration).

\begin{lem}
Let $w\in W$. Then $R_w B\cong B$. 
\end{lem}  

\begin{proof}
It suffices to show that $R_q B\cong B$ for all $q\in S$. We have 
$$R_s B_{tst} B_s B_t\cong B_{tst} R_s B_s B_t\cong B_{tst} B_s B_t\cong B,$$
where the first isomorphism follows from Relation~\eqref{eq:2} and the second one from Relation~\eqref{eq:1}. Similarly we have 
$$R_t B_{tst} B_s B_t\cong B_s R_t B_s B_t\cong B_s B_{tst} B_t\cong B_{tst} B_s B_t,$$
where the first isomorphism follows from Relation~\eqref{eq:2}, the second one from both~\eqref{eq:1} and \eqref{eq:2}, and the last one uses the fact that $B_{tst} B_s\cong B_s B_{tst}$ because $s$ and $tst$ commute (which follows for instance from Lemma~\ref{Soergel}). 
\end{proof}

It follows that if $B$ was decomposable, then tensoring on the left by any $R_w$ would permute the various summands, i.e., if $B_1$ is a summand, then $R_w B_1$ is also a summand. Using this property (for various $w$) together with the fact that a direct summand of a bimodule with a standard filtration also inherits a standard filtration, it is easy to check that $B$ is indecomposable. Now the only bimodule of the form $R(A)$ which has a filtration where each $R_x$ for all $x \in W$ appears as subquotient is $R(W)$, and comparing the graded dimensions of $R(W)$ and $B$ one sees that they do not coincide. This shows that there is an indecomposable bimodule in $\mathcal{C}$ which is not isomorphic to a ring of regular functions on a union of twisted diagonals.

It also shows that $B$ cannot be isomorphic to a (shift of a) bimodule of the form $R_w B' R_{w'}$ (see Remark~\ref{rmk:comb}), where $B'\in\mathcal{B}$ is an indecomposable Soergel bimodule and $w, w'\in W$, since by \cite[Section 4]{S} every indecomposable bimodule in $\mathcal{B}$ (and hence using Relations~\eqref{eq:4} every twist of it by $R_w$) is isomorphic to $R(A)$ for some $A\subseteq W$. 

\subsection{Type $A_3$}

Let $W$ be of type $A_3$, with $S = \{s,t,u\}$ such that $su = us$. 

It is well-known that already in $\mathcal{B}$, there are indecomposable bimodules which are not isomorphic to $R(A)$ for some $A\subseteq W$. For instance, the indecomposable Soergel bimodule $B$ associated to the group element $tsut$ has two shifted copies of $R_e$ in its standard filtration (equivalently the Kazhdan-Luzstig polynomial $h_{e, tust}$ is not a monomial) while an indecomposable Soergel bimodule of the form $R(A)$ has exactly one shifted copy of $R_x$, $x\in A$ appearing in its standard filtration. 

Consider the subcategory $\mathcal{C}'\subseteq\mathcal{C}$ generated by $B_s, B_t, B_u$ and $B_{sts}$. The subcategory $\mathcal{C}_1\subseteq \mathcal{C}'$ generated by $B_s, B_t$ and $B_{sts}$ is equivalent to the category of type $A_2$ which we described above, but the subcategory $\mathcal{C}_2$ generated by $B_{sts}, B_t$ and $B_u$ is not: for instance we have $B_t B_u B_t \not\cong B_t B_{sts} B_t$. Note that using Relations~\eqref{eq:1} and \eqref{eq:2} one sees that every subcategory of $\mathcal{C}_2$ generated by two of the generators is equivalent to a Soergel category of type $A_2$. It follows that $\langle \mathcal{C}_2\rangle$ is again a quotient of the affine Hecke algebra of type $\widetilde{A_2}$, but it is not even clear to us whether the algebra $\langle \mathcal{C}_2 \rangle$ has finite rank or not (equivalently if there are finitely many indecomposables in $\mathcal{C}_2$). 

As in type $B_2$, we give an example of an indecomposable bimodule $B$ in $\mathcal{C}$ which cannot be isomorphic to a (shift of a) bimodule of the form $R_w B' R_{w'}$, where $B'$ is an indecomposable bimodule in $\mathcal{B}$. Consider $B:= B_s R_t B_u \cong B_s B_{tut} R_t$. This element has a filtration with subquotients given by (shifts of) $R_t$, $R_{st}$, $R_{tu}$ and $R_{stu}$, each appearing exactly once. Now the only indecomposable Soergel bimodules $B'\in\mathcal{B}$ with exactly four subquotients in their (twisted) standard filtrations are the $B'$ associated to the group elements $x=s_1s_2$, $s_i\in S$, $s_1\neq s_2$ (and up to shifts the subquotients are $R, R_{s_1}, R_{s_2}, R_{s_1 s_2}$; the subquotients are independent of the filtration but the shifts may differ). Set $A:=\{t, st, tu, stu\}$. Assume that $B\cong R_w B' R_{w'}\cong R_{\widetilde{w}} \underbrace{R_{w'^{-1}} B' R_{w'}}_{=:B''}$ for $B'\in\mathcal{B}$ associated to such a group element $x$. Then $B''$ has a filtration with four subquotients given by (possibly shifted) $R_y$ for $y=1, t_1, t_2, t_1t_2$, where $t_i$ are reflections ($t_i:=w'^{-1} s_i w'$). It implies that $\widetilde{w} \{ 1, t_1, t_2, t_1t_2\}=A$, in particular, that $\widetilde{w}\in A$. Checking with the four possible values of $\widetilde{w}$, we see that $\widetilde{w}^{-1} A$ is never of the form $\{1, t_1, t_2, t_1t_2\}$, where $t_i$ are reflections, which concludes.  

\vspace{1cm}

\paragraph*{\bf Acknowledgements.}~The authors are grateful to the Hausdorff Research Institute for Mathematics (HIM), University of Bonn, for the warm hospitality and the support as this project was completed during their visit at the HIM in fall 2017 within the Trimester Program 'Symplectic Geometry and Representation Theory'. They also thank the GRDI 'Representation theory' and the DFG for funding the stay of the second author in Nancy in September 2017. The first author was funded by the ANR GeoLie ANR-15-CE40-0012.  

The authors thank Pierre-Emmanuel Chaput, Emmanuel Wagner and Geordie Williamson for useful discussions and suggestions.

\end{document}